\documentclass[12pt]{amsart}
\usepackage{amsmath,amsthm,amsfonts,latexsym,amssymb,amscd,color}
\usepackage{chemarr}
\newtheorem{thm}{Theorem}[section] %[subsection]

\newtheorem{lm}[thm]{Lemma}

\newtheorem{clm}[thm]{Claim}
\newtheorem*{clm*}{Claim}
\theoremstyle{definition}

\numberwithin{equation}{section}
\newcommand{\cproof}{\noindent{\it Proof of Claim.}\ } %used in spf environment
\newcommand{\cqed}{\hfill\rule{1.3mm}{3mm}}

\DeclareMathOperator{\Sub}{\mathsf{Sub}}

\DeclareMathOperator{\Con}{Con}

\begin{document}

\title[A note on $\mathbb M_{p+1}$]{A note on ``A minimal
  congruence \\ lattice representation for $\mathbb M_{p+1}$''}

\author{Keith A. Kearnes}
\address[Keith A. Kearnes]{Department of Mathematics\\
University of Colorado\\
Boulder, CO 80309-0395\\
USA}
\email{kearnes@colorado.edu}



\begin{abstract}
We reprove a theorem from \cite{BGIT}
%Bunn, Grow, Insall, and Thiem
which asserts that a minimal congruence lattice
representation for $\mathbb M_{p+1}$ has size $2p$,
and is an expansion of a regular $D_{2p}$-set.
\end{abstract}

\maketitle

\setcounter{section}{1}

%\section{}\label{}
I recently had the opportunity to review \cite{BGIT}
for Mathematical Reviews. This paper describes a minimal
congruence lattice representation for $\mathbb M_{p+1}$
via an argument involving
automorphism groups of regular graphs.
After filing my review, I realized that there is a 
short, group-theoretic proof of the result. It is based
on the following:

\begin{lm} \label{lemma}
Let $G$ be a finite group and let $H$ be a subgroup.
Suppose that the subgroup interval $I[H,G]\leq \Sub(G)$
is isomorphic to $\mathbb M_n$.
If $[G:H] < 2n$, then
\begin{enumerate}
\item $H\lhd G$.
\item $G/H$ is a dihedral group.
%  \cong D_{2p}$ (= the symmetry group of the regular $p$-gon).
\item $n=p+1$ for some prime $p$, and $G/H\cong D_{2p}$.
%\item $\Con(\langle G/H; G\rangle) \cong \Sub(D_{2p})\cong \mathbb M_{p+1}$.
\end{enumerate}
\end{lm}

\begin{proof}
Assume that the intermediate subgroups between $H$ and $G$
are $K_1,\ldots, K_n$ with $[K_1:H]\leq \cdots \leq [K_n:H]$.

\begin{clm} \label{claim}
$[K_1:H]=[K_2:H]=2$. Hence $H\lhd K_1$ and $H\lhd K_2$.
\end{clm}

Since $K_i\cap K_j=H$ when $i\neq j$,
the only coset of $H$ contained in more than one $K_i$ is $H$
itself. We are claiming that $K_1$ and $K_2$ each contain
exactly one other coset of $H$. If this is false,
then $K_1\setminus H$ contains at least one coset of $H$,
say $k_1H$, while 
$K_i\setminus H$ contains at least
two cosets of $H$ for $i>1$, say $k_{i,1}H$ and $k_{i,2}H$.
This is already too many cosets of $H$, since 
\[
\underbrace{1}_{\textrm{coset $H$}}+
\underbrace{1}_{\textrm{coset $k_1H$}}+
\underbrace{2(n-1)}_{\textrm{cosets $k_{i,1}H$ and $k_{i,2}H$}} = 2n,
\]
and we have assumed $[G:H]<2n$.
$\Box$

Now
%we use Claim~\ref{claim} to complete the proof.
$H\lhd K_1$ and $H\lhd K_2$,
so $H\lhd (K_1\vee K_2) = G$, establishing Item (1).
The group $G/H = (K_1/H)\vee (K_2/H)$ is generated by
two $2$-element subgroups, hence 
by two involutions. This implies that $G/H$ is 
a dihedral group, say
$D_{2m}$, establishing Item (2).
If the rotation subgroup $R\leq D_{2m}$ 
were not simple, then there would exist
a height-3 chain $\{1\}<S<R<D_{2m}$
in the height-2 lattice
$\Sub(D_{2m})\cong \Sub(G/H)\cong \mathbb M_n$. Consequently
$R$ must be simple, so $D_{2m}=D_{2p}$
for some prime $p$. Now
$\mathbb M_n\cong \Sub(G/H)\cong \Sub(D_{2p})\cong \mathbb M_{p+1}$,
so $n=p+1$ for some prime. This establishes Item (3).
\end{proof}

Lemma~\ref{lemma} implies that any ``sufficiently small''
(meaning: size $<2n$)
congruence lattice 
representation of $\mathbb M_n$ arises from
a regular dihedral action.
%, as is explained in the next theorem.

\begin{thm} (Compare with \cite[Theorem~4.6]{BGIT})
Assume that $p$ is a prime.
\begin{enumerate}
\item  If $\mathbb A$
  is a faithful, transitive $G$-set,
  $\Con(\mathbb A)\cong \mathbb M_{p+1}$, and $|A|<2(p+1)$,
  then $G\cong D_{2p}$, $|A| = 2p$, and $\mathbb A$ is a regular $D_{2p}$-set.
  A regular $D_{2p}$-set
is an example of such $\mathbb A$.
\item  If $\mathbb A$ is a finite algebra satisfying
  $\Con(\mathbb A)\cong \mathbb M_{p+1}$ where $p$
  is an odd prime, then $|A|\geq 2p$.  
\end{enumerate}
Hence, a minimal congruence lattice
representation for $\mathbb M_{p+1}$ has size $2p$,
and must be an expansion of a regular $D_{2p}$-set.
\end{thm}

\begin{proof}
For Item (1), we may assume that 
$\mathbb A = \langle G/H; G\rangle$ where $H$ is a core-free
subgroup of $G$. We have $I[H,G]\cong \Con(\mathbb A)\cong \mathbb M_{p+1}$,
and $[G:H]=|A|<2(p+1)$, so by
Lemma~\ref{lemma} in the case $n=p+1$,
and by
the core-freeness of $H$,
it follows that $H=\{1\}$, $G\cong D_{2p}$,
$|A|=|G|=|D_{2p}|=2p$, and $\mathbb A$ is a regular $D_{2p}$-set.
Any regular $D_{2p}$-set meets all of the conditions on $\mathbb A$.

Item (1) shows
that it is possible to find an algebra $\mathbb A$ representing
$\mathbb M_{p+1}$ where $|A|=2p$. We now argue that
any representation satisfying $|A|\leq 2p$ satisfies $|A|=2p$.
When $n>3$ the lattice $\mathbb L=\mathbb M_n$ satisfies
the ``P\'{a}lfy-Pudl\'{a}k Properties (A), (B), (C)'' from \cite{ppp},
which guarantee that if $\mathbb A$ is a finite algebra such that
$\Con(\mathbb A)\cong \mathbb L$ is a minimal
congruence lattice representation, then $\mathbb A$
is an expansion of a faithful, transitive $G$-set $\mathbb A^{\circ}$
that is also a minimal representation of $\mathbb L$.
%which satisfies $\Con(\mathbb A^{\circ})\cong \mathbb L$.
When $p$ is an odd prime and $n=p+1$ we do have $n>3$,
so any minimal
%congruence lattice
representation $\mathbb A$ for $\mathbb M_{p+1}$
is an expansion of a faithful, transitive $G$-set $\mathbb A^{\circ}$
that is also a minimal
%congruence lattice
representation for $\mathbb M_{p+1}$.
We are assuming that $|A|\leq 2p \;\; (< 2(p+1))$, so 
applying Item (1) to the reduct $\mathbb A^{\circ}$ we get $|A|=2p$
and $\mathbb A$ is an expansion of a regular $D_{2p}$-set $\mathbb A^{\circ}$.
\end{proof}
  
\bibliographystyle{plain}

\begin{thebibliography}{10}

\bibitem{BGIT}
  Bunn, Roger; Grow, David; Insall, Matt; Thiem, Philip,
  {\it A minimal congruence lattice representation for $\mathbb M_{p+1}$},
J.\ Aust.\ Math.\ Soc.\ {\bf 108} (2020), 332--340.

\bibitem{ppp}
P\'{a}lfy, P\'{e}ter P\'{a}l; Pudl\'{a}k, Pavel,
  {\it Congruence lattices of finite algebras
    and intervals in subgroup lattices
    of finite groups},
  Algebra Universalis {\bf 11} (1980), 22--27.

\end{thebibliography}

\end{document}